\documentclass[a4paper]{article}
\usepackage[margin=3cm]{geometry}

\usepackage{amsmath, amsfonts, amsthm, amssymb, bbm, bm}
\usepackage{mathrsfs}
\usepackage{centernot}
\usepackage{enumerate}
\usepackage{tabularx}
\usepackage{multicol}
\usepackage{multirow}

\newcommand{\E}{\mathbb{E}}

\newcommand{\G}{\mathcal{G}}
\newcommand{\X}{\mathfrak{X}}

\DeclareMathOperator{\pa}{pa}

\DeclareMathOperator{\an}{an}

\DeclareMathOperator{\dec}{de}

\DeclareMathOperator{\sib}{sib}
\DeclareMathOperator{\dis}{dis}

\DeclareMathOperator{\mbl}{mb}

\DeclareMathOperator{\pre}{pre}

\newcommand{\cmid}{\,|\,}

\newcommand{\lrarr}{\leftrightarrow}
\newcommand{\sto}{*\!\!\!\to}
\newcommand{\getss}{\gets\!\!\!*}

\newcommand\indep{\protect\mathpalette{\protect\independenT}{\perp}}
\def\independenT#1#2{\mathrel{\rlap{$#1#2$}\mkern2mu{#1#2}}}

\theoremstyle{plain}
\newtheorem{lem}{Lemma}[section]
\newtheorem{thm}[lem]{Theorem}
\newtheorem{prop}[lem]{Proposition}

\theoremstyle{definition}
\newtheorem{dfn}[lem]{Definition}
\newtheorem{rmk}[lem]{Remark}
\newtheorem{exm}[lem]{Example}

\newcommand{\benum}{\begin{enumerate}}
\newcommand{\eenum}{\end{enumerate}}

\newcommand{\bitem}{\begin{itemize}}
\newcommand{\eitem}{\end{itemize}}

\newcommand{\barr}{\begin{array}}
\newcommand{\earr}{\end{array}}

\newcommand{\bmat}{\begin{pmatrix}}
\newcommand{\emat}{\end{pmatrix}}

\newcommand{\blist}{\renewcommand{\labelenumi}{\textbf{\arabic{enumi}}.} \begin{enumerate}}
\newcommand{\elist}{\end{enumerate} \renewcommand{\labelenumi}{\arabic{enumi}.}}

\def\bal#1\eal{\begin{align*}#1\end{align*}}

\usepackage{tikz}
\usetikzlibrary{shapes, arrows, arrows.meta, calc, positioning}

\tikzset{nv/.style={circle, color=red, fill=red, inner sep=0.5mm}}
\tikzset{rv/.style={circle, draw, thick, minimum size=7mm, inner sep=0.5mm}}
\tikzset{fv/.style={rectangle, draw, thick, minimum size=7mm, inner sep=0.5mm}}
\tikzset{lv/.style={circle, color=red, fill=gray!30, draw, thick, minimum size=7mm, inner sep=0.5mm}}
\tikzset{rve/.style={ellipse, draw, thick, minimum size=7mm, inner sep=0.5mm}}

\tikzset{rvs/.style={circle, draw, thick, minimum size=6mm, inner sep=0.5mm}}
\tikzset{fvs/.style={rectangle, draw, thick, minimum size=6mm, inner sep=0.5mm}}
\tikzset{lvs/.style={circle, color=red, fill=gray!30, draw, thick, minimum size=6mm, inner sep=0.5mm}}
\tikzset{rves/.style={ellipse, draw, thick, minimum size=6mm, inner sep=0.5mm}}

\tikzset{deg/.style={->, very thick, color=blue}}
\tikzset{degl/.style={->, very thick, color=red}}
\tikzset{beg/.style={<->, very thick, color=red}}
\tikzset{cdeg/.style={{Circle[length=+2pt 2.5,width=+2pt 2.5, fill=none]}->, very thick, color=blue}}
\tikzset{cceg/.style={{Circle[length=+2pt 2.5,width=+2pt 2.5, fill=none]}-{Circle[length=+2pt 2.5,width=+2pt 2.5, fill=none]}, very thick}}
\tikzset{uceg/.style={{Circle[length=+2pt 2.5,width=+2pt 2.5, fill=none]}-, very thick}}
\tikzset{ueg/.style={very thick}}

\newcommand{\del}{\mathbin{\not\!\|\,}}

\usepackage[round]{natbib}

\usepackage[symbol]{footmisc}

\usepackage{tikz}
\usetikzlibrary{shapes,arrows,positioning}

\tikzset{
    position/.style args={#1:#2 from #3}{
        at=(#3.#1), anchor=#1+180, shift=(#1:#2)
    }
}
\tikzset{rv/.style={circle, draw, thick, minimum size=6.5mm, inner sep=0.5mm}}
\tikzset{fv/.style={rectangle, draw, thick, minimum size=6mm, inner sep=0.5mm}}

\title{Latent-Free Equivalent mDAGs}
\author{Robin J.~Evans}
\date{\today}

\newcommand{\M}{\mathcal{M}}

\begin{document}

\maketitle

\begin{abstract}
We show that the marginal model for a discrete directed acyclic graph
(DAG) with hidden variables is distributionally equivalent to another 
fully observable DAG model if and only if it does not induce any non-trivial 
inequality constraints.
\end{abstract}
 
\section{Introduction}

The \emph{marginal model} of a directed acyclic graph (DAG) model with latent 
variables is defined simply as the set of distributions that are realizable 
as margins over the observed variables, from those joint distributions that
are Markov with respect 
to the whole graph and where no restrictions are placed on the 
state-space of the latents.  It was shown by \citet{evans16mdags}
that we can represent this class of models using a collection of hypergraphs 
known as mDAGs (standing for \emph{marginal} DAGs).  

Much is known about the properties of these models.  For example, in the discrete 
and Gaussian cases the models are semi-algebraic, meaning that the equalities
and inequalities that define them are all polynomials in the joint probabilities
or covariance matrix respectively.  The equality constraints in the discrete case 
are understood \citep{evans18complete},
and there are methods for finding (in principle) all inequality constraints
as well \citep{wolfe19inflation, navascues20inflation}.  However, it is 
still an important open problem to determine whether or not two marginal
models are equivalent. 

A specific question that may be of interest in this respect, is whether or 
not the marginal
model of a DAG with observed variables $V$ and latent variables $L$ 
is \emph{distributionally equivalent} to another DAG over only $V$.  In other words,
is the set of distributions that is in the marginal model defined by a subset 
of variables in one DAG identical to the entire model defined by some other
DAG? 
The question of understanding \emph{distributional equivalence classes} 
of models is a significant open problem, and is a critical component of 
causal model search.  We cannot hope to choose between two models from data 
if they are distributionally equivalent, so any contribution to 
understanding when this occurs is extremely useful.  In addition, for the 
purpose of finding the most efficient influence function in semiparametric 
statistics, for example, this is much easier if the model is known
to be (equivalent to) a DAG model, because the tangent cone can be
easily decomposed into pieces that correspond to each variable conditional
precisely upon its parents \citep[Section 4.4]{tsiatis06semiparametric}.  

We show in this paper that, if the observed variables are all discrete, 
this is true if and only if the marginal model does not 
induce any inequality constraints, beyond those already implied by the 
required equality constrains and the necessity of probabilities being 
non-negative.  This main result is stated in the following theorem; note 
that $\M(\G)$ denotes the collection of distributions that satisfy the 
marginal Markov property (Definition \ref{dfn:marginal}) with respect to 
the mDAG $\G$ (Definition \ref{dfn:mDAG}). 

\begin{thm} \label{thm:main}
Let $\G$ be an mDAG with vertices $V$, inducing a model $\M(\G)$ over a collection 
of discrete random variables $X_V$.  Then there exists a DAG $\mathcal{H}$ 
such that $\M(\G) = \M(\mathcal{H})$ if and only if $\M(\G)$ is described entirely 
by probability distributions that satisfy a finite number of equality constraints. 
\end{thm}

The `only if' direction is trivial, since DAG models do not
imply any inequalities, and are defined entirely by a finite list of ordinary 
conditional independences.  

In Section \ref{sec:graphs} we present necessary concepts relating to 
DAGs and mDAGs, including distributional equivalence.  In Section 
\ref{sec:nested} we introduce the `nested' Markov model, and show that any 
model with a non-trivial nested constraint can be reduced to a model with only 
standard conditional independences that 
are not consistent with any DAG.  In Section \ref{sec:main} we prove our 
main result, and in Section \ref{sec:cont} we consider possible extensions 
to continuous random variables. 

\section{Basics concepts for DAGs and mDAGs} \label{sec:graphs}

We consider mixed (hyper)graphs with one set of vertices $V$, and (up to) 
two edge sets $\mathcal{D}$ and $\mathcal{B}$; the set $\mathcal{D}$ 
contains ordered pairs of vertices, and $\mathcal{B}$ is a simplicial 
complex over the set $V$.

\begin{dfn} \label{dfn:mDAG}
In a \emph{directed graph} $\G = (V,\mathcal{D})$, if $(v,w) \in \mathcal{D}$ 
then we write $v \to w$ and say that $v$ is a \emph{parent} of $w$, and $w$ a  
\emph{child} of $v$.  The set of parents of $w$ in $\G$ is denoted by 
$\pa_\G(w)$.  A \emph{directed walk with length $k$} is 
a sequence of vertices $v_0,\ldots,v_k$ such that each $v_i$ is a parent 
of $v_{i+1}$.  A directed graph is said to be \emph{acyclic} if there 
are no directed walks of length $k \geq 1$ from any vertex back to itself; 
we call such an object a \emph{directed acyclic graph} (DAG). 

An \emph{mDAG} is a DAG $(V,\mathcal{D})$ together with a simplicial 
complex $\mathcal{B}$ over $V$.  
We refer to the entries of $\mathcal{B}$ as \emph{bidirected faces}, and 
the maximal entries as \emph{bidirected facets}.  If a face contains two 
vertices we may also call it a \emph{bidirected edge}.
\end{dfn}

An example of an mDAG consisting of a DAG with 4 edges and the bidirected facets 
$\{a,b\}$, $\{a,c,e\}$ and $\{d,f\}$ is shown in Figure \ref{fig:mDAG}(i).  Note 
that we use blue to draw directed edges, and red for the bidirected facets.  

\subsection{Marginal models}

We first define what it means for a distribution to be Markov with respect 
to a DAG.

\begin{dfn} \label{dfn:lmp}
A distribution $p$ over random variables $X_V$ is said to be 
\emph{Markov} with respect to a directed acyclic graph $\G$ if
there is a topological ordering $\prec$ of $V$ such that
\begin{equation*}
    X_v \indep X_{\pre(v; \prec) \setminus \pa(v)} \mid X_{\pa(v)}  \text{ under }p
\end{equation*}
for each $v \in V$, where $\pre_\G(v; \prec) = \{w \in V : w \prec v\}$.
\end{dfn}

Note that we omit the subscripts on operators when they are themselves
written in a subscript and the meaning is clear. 
We remark that if Definition \ref{dfn:lmp} holds for one topological
ordering, then it can be shown using standard implications of 
conditional independences that 
it holds for every other topological ordering \citep{lauritzen90indep}.

%
%
%
%

Let $\G$ be an mDAG, and let $\overline{\G}$ denote
the \emph{canonical DAG} for $\G$.  That is, we replace 
each bidirected facet $B$ with a latent variable that has the
set of children $B$; see Figure \ref{fig:mDAG}(ii) for the canonical DAG 
associated with the mDAG in \ref{fig:mDAG}(i).  We colour the edges similarly
in the mDAG: if an edge is between two observed vertices it is blue, and
otherwise it is red.

 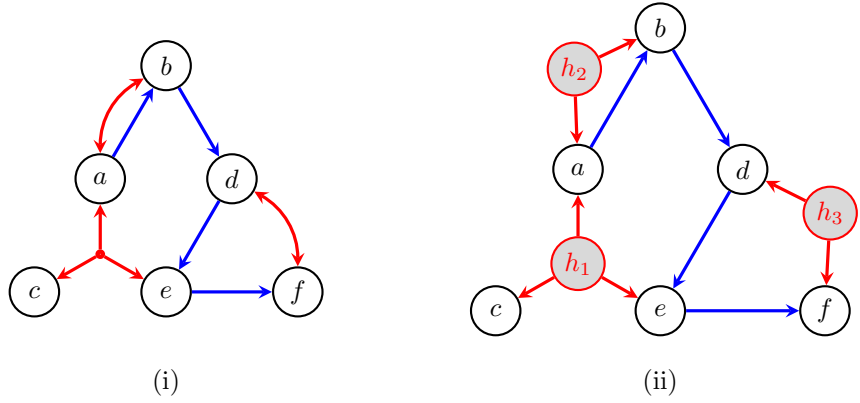
\begin{figure}
     \centering
 \begin{tikzpicture}
 [>=stealth, node distance=20mm]
 \pgfsetarrows{latex-latex}
 \begin{scope}
 \node (0) {};
  \node[rv] at (150:1) (1a) {$a$};
   \node[rv] at (90:2) (2a) {$b$};
  \node[rv] at (30:1) (4a) {$d$};
  \node[rv] at (270:1) (5a) {$e$};
  \node[rv] at (210:2) (3a) {$c$};
  \node[rv] at (330:2) (6) {$f$};
   \node[draw, circle, color=red, inner sep=0.25mm, ultra thick] (h) at (210:1.0) {};
  \draw[deg] (2a) -- (4a);
  \draw[deg] (1a) -- (2a);
  \draw[deg] (5a) -- (6);
  \draw[deg] (4a) -- (5a);
  \draw[degl] (h) -- (3a);
  \draw[degl] (h) -- (5a);
  \draw[degl] (h) -- (1a);
  \draw[beg] (4a) to[bend left] (6);
 \draw[beg] (1a) to[bend left] (2a);
 \node[below of=5a, yshift=8mm] {(i)};
 \end{scope}
 \begin{scope}[xshift=6.5cm]
 \node (0) {};
  \node[rv] at (150:1.25) (1a) {$a$};
   \node[rv] at (90:2.5) (2a) {$b$};
  \node[rv] at (30:1.25) (4a) {$d$};
  \node[rv] at (270:1.25) (5a) {$e$};
  \node[rv] at (210:2.5) (3a) {$c$};
  \node[rv] at (330:2.5) (6) {$f$};
   \node[lv] (h1) at (210:1.25) {$h_1$};
   \node[lv, position=60:.4 from 1a, xshift=-6mm, yshift=3.75mm] (h2) {$h_2$};
   \node[lv, position=-60:.4 from 4a, xshift=6mm, yshift=3.75mm] (h3) {$h_3$};
  \draw[deg] (2a) -- (4a);
  \draw[deg] (1a) -- (2a);
  \draw[deg] (4a) -- (5a);
  \draw[deg] (5a) -- (6);
  \draw[degl] (h1) -- (3a);
  \draw[degl] (h1) -- (5a);
  \draw[degl] (h1) -- (1a);
  \draw[degl] (h2) -- (1a);
  \draw[degl] (h2) -- (2a);
  \draw[degl] (h3) -- (4a);
  \draw[degl] (h3) -- (6);
 \node[below of=5a, yshift=10.5mm] {(ii)};
 \end{scope}
 \end{tikzpicture}
     \caption{(i) An mDAG and (ii) its canonical DAG.  Note that (i) is also the 
     latent projection of (ii) over $\{a,b,c,d,e,f\}$.} 
     \label{fig:mDAG}
 \end{figure}

\begin{dfn} \label{dfn:marginal}
We define the \emph{marginal model} for $\G$ as the set of
distributions that can be obtained as a margin over the observed 
variables in $\G$ under a distribution that is Markov with respect 
to $\overline{\G}$.  This set of distributions is denoted $\M(\G)$.
\end{dfn}

This model is defined in \citet{evans16mdags}, and its 
properties and the sufficiency of mDAGs for representing 
such models are laid out more fully in that paper.
We remark that the state-space of the latent variables is
in principle arbitrary, but that a uniform random variable 
on $(0,1)$ always has sufficiently large cardinality.  A result of
\citet{rosset18universal} shows that if all the variables are 
discrete with a finite state-space, then there is a corresponding
finite bound on the cardinality of the latent variables.

\subsection{Distributional and Markov equivalence}

Given an mDAG, one can read off the conditional independences that are 
satisfied by distributions that are Markov to it using \emph{m-separation}.  
For readers familiar with d-separation in directed graphs, it is essentially
the same; like d-separation it is based on whether there is an \emph{open path}
between two variables, or whether all such paths are \emph{blocked}.  The only 
modification is that the definition of a collider and non-collider has to be 
expanded
to take account of bidirected facets.  The full definition is given in Appendix 
\ref{sec:msep}.  

\begin{dfn}
We say that two mDAGs $\G$ and $\G'$ are \emph{distributionally 
equivalent} if $\M(\G) = \M(\G')$.  The ordinary conditional independences
implied by $\M(\G)$ are used to define the \emph{ordinary Markov} model
for $\G$.  We say that two mDAGs are \emph{ordinary Markov equivalent}
if they imply the same set of conditional independences (i.e.~they exhibit
the same collection of m-separations.)
\end{dfn}

If two graphs are distributionally equivalent then they are also ordinary
Markov equivalent.  See Proposition \ref{prop:subequiv} for a comparison 
between these two models, as well as the `nested' Markov model (see Section 
\ref{sec:nested}).

\begin{exm} \label{exm:const}
Consider the mDAG $\G$ shown in Figure \ref{fig:exm}(i).  We can see that 
$a \perp_m c \mid b$ and so therefore if $p \in \M(\G)$ it holds that
$X_a \indep X_c \mid X_b$.  Note that there is no way to m-separate 
$a$ and $d$ in this graph, because there is a directed path via $b$ 
and $c$, and if we condition on either of these vertices then
a path $a \to b \lrarr d$ will be opened up.  

In fact there \emph{is} a constraint between $X_a$ and $X_d$, but it is only
revealed after fixing the vertex $c$ (see Sections \ref{sec:nested} and \ref{sec:fix} for more detail); this yields the graph in (ii), 
which shows that now $d \perp_m a \mid c$, so there is a \emph{nested} constraint:
$X_d \indep X_a \mid X_c$ after fixing $X_c \mid X_b$.

In addition, the model implied by the graph in Figure \ref{fig:exm}(ii) 
contains an inequality 
constraint, being the Clauser-Horne-Shimony-Holt (CHSH) inequality 
\citep{clauser69proposed}.  This says that 
if (for example) all four variables take values in $\{-1,+1\}$, then
\begin{align}
    -2 &\leq \E[ X_b X_d \mid X_a = -1, X_c = +1] + \E[ X_b X_d \mid X_a = +1, X_c = -1]\nonumber\\
    & \quad \mathbin{+} \E[ X_b X_d \mid X_a = -1, X_c = -1] - \E[ X_b X_d \mid X_a = +1, X_c = +1] \leq 2.  \label{eqn:chsh}
\end{align}
Note however that a distribution exists satisfying the two independences 
given, but for which this quantity in (\ref{eqn:chsh}) attains the value 4: set $P(X_b = -X_d = \pm 1) = \frac{1}{2}$ if $X_a = X_c = +1$, and $P(X_b = X_d = \pm 1) = \frac{1}{2}$ otherwise, 
and one can verify that the two independences mentioned are satisfied, and that each
term in (\ref{eqn:chsh}) has the value +1.  In this sense
the inequalities are \emph{non-trivial}, because they are not implied by any 
of the equality constraints.

\begin{figure}
    \centering
\begin{tikzpicture}[>=stealth, node distance=17.5mm]
    \begin{scope}
        \node[rv] (1) {$a$};
        \node[rv, right of=1] (2) {$b$};
        \node[rv, right of=2] (3) {$c$};
        \node[rv, right of=3] (4) {$d$};
        \draw[deg] (1) to (2);
        \draw[deg] (2) to (3);
        \draw[deg] (3) to (4);
        \draw[beg] (2) to[bend left] (4);
        \node[below of=2, yshift=8mm, xshift=8.75mm] {(i)};
    \end{scope}
    \begin{scope}[xshift=8cm]
        \node[rv] (1) {$a$};
        \node[rv, right of=1] (2) {$b$};
        \node[rv, right of=2] (3) {$c$};
        \node[rv, right of=3] (4) {$d$};
        \draw[deg] (1) to (2);
        \draw[deg] (3) to (4);
        \draw[beg] (2) to[bend left] (4);
        \node[below of=2, yshift=8mm, xshift=8.75mm] {(ii)};
    \end{scope}
    \begin{scope}[yshift=-3cm, xshift=4cm]
        \node[rv] (1) {$a$};
        \node[rv, right of=1] (2) {$b$};
        \node[rv, right of=2] (3) {$c$};
        \node[rv, right of=3] (4) {$d$};
        \node[lv, inner sep=0.1mm, above of=3, yshift=-7.5mm] (L) {$h$};
        \draw[deg] (1) to (2);
        \draw[deg] (2) to (3);
        \draw[deg] (3) to (4);
        \draw[degl] (L) to (4);
        \draw[degl] (L) to (2);
        \node[below of=2, yshift=8mm, xshift=8.75mm] {(iii)};
    \end{scope}

\end{tikzpicture}
    \caption{(i) An mDAG exhibiting all three kinds of constraint: a 
    conditional independence ($X_a \indep X_c \mid X_b$), a nested conditional
    independence ($X_d \indep X_a \mid X_c$ after fixing $X_c \mid X_b$), and 
    an inequality constraint (see Example \ref{exm:const}).  (ii) The 
    graph from (i) after $c$ has been fixed.  (iii) The canonical DAG for the
    mDAG in (i).}
    \label{fig:exm}
\end{figure}
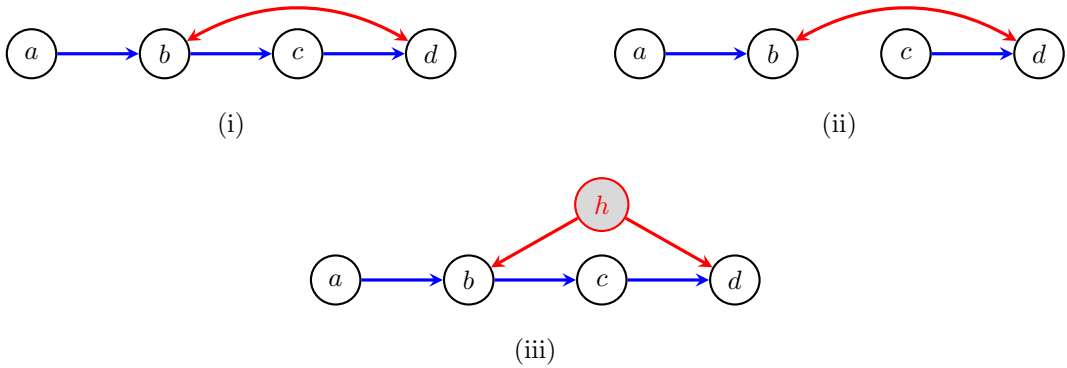

\end{exm}

\begin{rmk}
We can also read off some inequalities using a generalization of m-separation called
\emph{e-separation} \citep{evans:12}; this involves first deleting a set of variables 
$D$, and then checking for m-separation in the resulting graph.  If $A \perp_m B \mid C$ 
in the graph $\G$ after removing the vertices $D$ (and all edges incident to vertices in $D$) 
we denote it by $A \perp_e B \mid C \del D$; see 
Appendix \ref{sec:rvci} for more details on the resulting constraints. 
The \emph{instrumental inequality} of \citet{pearl95testability} can be
read off using this criterion, although (\ref{eqn:chsh}) cannot. 
\end{rmk}

\subsection{Equivalence}


Here we give some examples of (non-)equivalence of the marginal models for different
mDAGs.  Consider the graphs in Figure \ref{fig:equiv}, which are all ordinary Markov 
equivalent; the mDAGs in (i) and (ii) can be shown to be equivalent to the DAG in (iii).  

For (i), note that the only constraint in (iii) is that $X_a, X_c \indep X_d$.  This
can clearly be achieved by (i) just by setting the implied latent variable to tell 
$a$ and $c$ what values they each take, and then pass this information onto $b$.  
Since $X_a$ and $X_c$ are determined jointly, this clearly allows any distribution such 
that the constraint holds to be attained in the model for the graph in (i). 

For (ii), first note that it is clearly 
equivalent for the (implied) latent variable between $b$ and $d$ to simply contain
the value of $X_d$.  Hence the edge between $b$ and $d$ can be the same as in (i) and (iii).
Then, similarly, the latent variable between $a$ and $c$ can just contain $X_a$, so 
again we can replace it with a directed edge as in (iii).  Now, for the final bidirected
edge between $b$ and $c$, note that $b$ needs to know what value $c$ will take; this can be
arranged by making the latent variable be a map telling $X_c$ what to do for each value of 
$X_a$.  If this information is passed to $b$, then (since it can see $X_a$ directly) it 
can compute what $X_c$ must be.  Hence, we obtain equivalence between the two models. 

The graph in (iv) is \emph{not} equivalent to the other three, because the induced subgraph
over $\{a,b,c\}$ implies an inequality constraint \citep{fritz12beyond,evans16mdags}. 

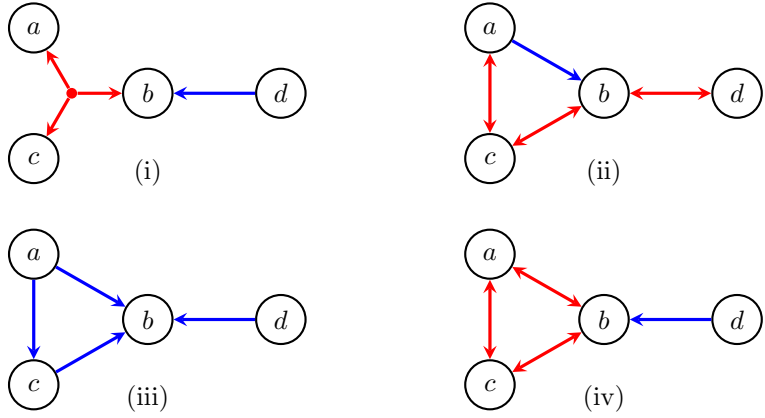
\begin{figure}
    \centering
\begin{tikzpicture}[>=stealth, node distance=17.5mm]
    \begin{scope}
        \node[nv] (0) {};
        \node[rv] (1) at (120:10mm) {$a$};
        \node[rv] (2) at (0:10mm)  {$b$};
        \node[rv] (3) at (240:10mm)  {$c$};
        \node[rv, right of=2] (4) {$d$};
        \draw[degl] (0) to (1);
        \draw[degl] (0) to (2);
        \draw[degl] (0) to (3);
        \draw[deg] (4) to (2);
        \node[below of=2, yshift=7mm, xshift=0mm] {(i)};
    \end{scope}
    \begin{scope}[xshift=6cm]
        \node (0) {};
        \node[rv] (1) at (120:10mm) {$a$};
        \node[rv] (2) at (0:10mm)  {$b$};
        \node[rv] (3) at (240:10mm)  {$c$};
        \node[rv, right of=2] (4) {$d$};
        \draw[deg] (1) to (2);
        \draw[beg] (1) to (3);
        \draw[beg] (3) to (2);
        \draw[beg] (2) to (4);
        \node[below of=2, yshift=7mm, xshift=0mm] {(ii)};
    \end{scope}
    \begin{scope}[xshift=0cm, yshift=-3cm]
        \node (0) {};
        \node[rv] (1) at (120:10mm) {$a$};
        \node[rv] (2) at (0:10mm)  {$b$};
        \node[rv] (3) at (240:10mm)  {$c$};
        \node[rv, right of=2] (4) {$d$};
        \draw[deg] (1) to (2);
        \draw[deg] (1) to (3);
        \draw[deg] (3) to (2);
        \draw[deg, <-] (2) to (4);
        \node[below of=2, yshift=7mm, xshift=0mm] {(iii)};
    \end{scope}
    \begin{scope}[xshift=6cm, yshift=-3cm]
        \node (0) {};
        \node[rv] (1) at (120:10mm) {$a$};
        \node[rv] (2) at (0:10mm)  {$b$};
        \node[rv] (3) at (240:10mm)  {$c$};
        \node[rv, right of=2] (4) {$d$};
        \draw[beg] (1) to (2);
        \draw[beg] (1) to (3);
        \draw[beg] (3) to (2);
        \draw[deg, <-] (2) to (4);
        \node[below of=2, yshift=7mm, xshift=0mm] {(iv)};
    \end{scope}

\end{tikzpicture}
    \caption{(i) and (ii) mDAGs that are equivalent to the DAG (iii). (iv) 
    is an mDAG that is \emph{not} equivalent to any DAG.}
    \label{fig:equiv}
\end{figure}

\section{Nested Markov model} \label{sec:nested}

As we have seen in Example \ref{exm:const}, there are two types of 
\emph{equality} constraint that can be obtained 
in an mDAG.  The first is an ordinary conditional independence, and the
second is a (strictly) \emph{nested} constraint, which is a conditional 
independence that arises only after probabilistically `fixing' some of the 
other variables.  We now define this operation more formally. 

\begin{dfn} \label{dfn:mb}
    The \emph{Markov blanket} of a vertex $v$ in an mDAG $\G$ is the set of 
    (other) vertices $w$ that can 
    be reached by a walk whose internal vertices are all \emph{colliders}, and 
    such that the first edge has an arrowhead into $v$; that is
    $w \to v$ or $w \lrarr \cdots \lrarr v$ or $w \to \lrarr \cdots \lrarr v$ 
    (where $w \lrarr v$ is shorthand for $v,w$ being contained in the same 
    bidirected facet.) 
    We denote this set by $\mbl_\G(v)$. 

    We say a vertex $v$ is \emph{fixable} (in $\G$) if it has
    no strict descendants (i.e.~vertices that can be reached by a directed walk 
    from $v$) that can also be reached by walks over only bidirected edges.

    Given such a vertex, we can \emph{fix} it in the graph by removing all 
    incoming edges (whether directed or bidirected), but keeping any directed
    edges oriented \emph{out} of $v$; let this new graph be $\G^*$.  
    Probabilistically, we compute 
    \begin{align*}
        p^*(x_V) &= \frac{p^*(x_v)}{p(x_v \cmid x_{\mbl(v)})} \cdot p(x_V),
    \end{align*}
    where $p^*(x_v)$ is an arbitrary strictly positive marginal density over $\X_v$.
\end{dfn}

Results from \citet{richardson23nested} tell us that if $p$ is in the marginal
model for $\G$, then $p^*$ will be in the marginal model for $\G^*$.  Hence any 
non-trivial constraints we deduce on $p^*$ must also apply to $p$.  Note
that the definition of a Markov blanket given here does not include paths 
beginning $v \to$, which is common in other papers; 
this is crucial in order to give the correct definition of the fixing operation. 

\subsection{Nested models are not DAG-like}

In this section we show that any non-trivial nested constraint in an 
mDAG $\G$ will imply that the conditional independence model after fixing cannot 
be represented by any DAG.  

\begin{prop} \label{prop:noDAGm}
Suppose that fixing a vertex $v$ from an mDAG $\G$ leads directly to a 
non-trivial nested constraint.  
Then the conditional independence model implied by $\G$ after the fixing is 
not faithfully represented by any DAG model. 
\end{prop}

\begin{proof}
%
When we fix $v$ we multiply by $p^*(x_v)/p(x_v \cmid x_{\mbl(v)})$, 
and so we artificially introduce the independence 
$X_v \indep X_{\mbl(v)}$ by performing the fixing. 
Let the new constraint be $X_A \indep X_B \mid X_C$, where each of $A$, 
$B$ and $C$ are chosen to be 
inclusion minimal; that is, if any vertex is removed from $A$ or $B$ then
the independence also held in some form before the fixing (possibly with
a different conditioning set), and if from $C$ 
then the required m-separation no longer holds in the new graph. 

Since the new constraint $X_A \indep X_B \mid X_C$ is non-trivial, 
it cannot have been induced just by deleting paths through $v$, so 
there exists a path $\pi$ from $a \in A$ to $b \in B$ \emph{not} through 
$v$, that was previously open given $A' \cup B' \cup C$, but is now 
blocked (here $A' = A\setminus \{a\}$ and $B' = B\setminus \{b\}$).  
Hence there is a set of colliders $S$ on $\pi$ that were ancestors of 
$v$ in $\G$, but are not after the fixing, and hence no longer ancestors 
of things in $A \cup B \cup C$.  

Then choose $D = A' \cup B' \cup C \cup S'$, where $S'$ is a maximal subset 
such that $a \perp_m b \mid A' \cup B' \cup C \cup S'$ in $\G^*$, but not 
if we add in another element $s \in S \setminus S'$.  Clearly 
$S \setminus S' \neq \emptyset$ 
from the discussion in the previous paragraph.  Now we can apply 
Proposition \ref{prop:noDAG} to obtain the result. 
%
%
\end{proof}

\section{mDAGs without nested constraints} \label{sec:main}

Now, we need only prove that models whose equality constraints are equivalent 
to those of an mDAG model (and not ordinary Markov equivalent to a conditional 
DAG model) will induce 
some sort of non-trivial inequality in their marginal model.  We can do this 
by assuming that we consider the `final' fixing to reveal a non-trivial
nested constraint, and then look at the independence model that this induces.

\subsection{Partial ancestral graphs}

If $\G$ does not have any nested constraints, then we consider its 
\emph{partial ancestral graph} (PAG) $[\G]$, which represents precisely the
ordinary conditional independence constraints implied by $\G$.  There is a 
one-to-one correspondence between PAGs and conditional independence models
induced by mDAGs \citep{richardson02ancestral, richardson03causal}. 
PAGs are ordinary mixed graphs (i.e.~they do not contain hyper-edges) with 
three edge markings: a tail, an arrowhead and a circle;  
a circle means that at least one \emph{maximal ancestral graph} (MAG) in the 
equivalence class has a tail mark here, and at least one has an 
arrowhead.  See Figures \ref{fig:bell} and
\ref{fig:disc} for some examples.  
More details about MAGs and PAGs are given in Appendices 
\ref{sec:anc_proj} and \ref{sec:pag}.  The crucial fact here is that
the conditional independence structure of any mDAG can always be 
represented by a MAG, and therefore by a PAG.

\begin{prop} \label{prop:zhang}
Suppose that $\mathcal{P} = [\G]$.  Then the conditional independence
structure of $\G$ is the same as that of a DAG if and only if $\mathcal{P}$ 
does not have any bidirected edges. 
\end{prop}

\begin{proof}
We know from Lemma 3.3.4 of \citet{zhang:thesis} that a PAG can always be
oriented to a MAG in such a way that it does not introduce any additional 
bidirected edges.  Hence, if there are none to start with, the model
is ordinary Markov equivalent to a DAG.  

For the converse, note that if it were false that would imply that 
the edge is bidirected in every Markov equivalent MAG, which 
contradicts the existence of a Markov equivalent DAG.
\end{proof}

Now, since the PAG represents \emph{invariant} edges (i.e.~ones that are 
the same in all members of the equivalence class), the graph is ordinary 
Markov equivalent to a DAG if and only if there are no bidirected edges in its PAG.

We say that a collider path $\langle v_0, \ldots, v_k\rangle$ is \emph{locally} 
unshielded if there is no edge between $v_i$ and $v_{i+2}$ for any 
$i=0,\ldots,k-2$.

\begin{prop} \label{prop:bidi_ineq}
Suppose that $\G$ contains no non-trivial nested constraints, and that there 
is a bidirected edge in $\mathcal{P} = [\G]$.  Then a non-trivial inequality 
constraint is induced over the distributions in $\M(\G)$.
\end{prop}

\begin{proof}
There are two reasons that a bidirected edge can be included in a PAG.  
Either there is a locally unshielded collider path of length 3 from (say) $a$ to 
$d$ (see Figure \ref{fig:bell}), or there is a discriminating path of length at 
least 3 (see Figure \ref{fig:disc}).  In the first case,
the PAG must have an induced subgraph of one of the forms in Figure 
\ref{fig:bell}.  The graphs in (i) and (ii) induce the CHSH 
inequality (\ref{eqn:chsh}) \citep{bell64einstein, clauser69proposed}.

For (iii) and (iv), consider the submodel in which all information 
about $X_d$ is contained as part of $X_c$.  This means that $X_d$ must obtain 
all its information from the latent it shares with $X_c$, since either $X_c$ or
$X_d$ is marginally independent of all other variables.  Hence we can remove 
the edge between $a$ and $d$, and then note that the mDAGs become distributionally 
equivalent to Figure \ref{fig:bell}(i).  Hence this submodel induces the CHSH 
inequality, and so the whole distribution also satisfies an inequality.
%

On the other hand, suppose that there is no locally unshielded collider path of
length $k \geq 3$ but there is a discriminating path 
$\langle a,v_1,\ldots,v_k,b,c\rangle$ with $k \geq 1$.  In fact, by 
Proposition \ref{prop:disc_path2}, if these conditions are satisfied, then there will 
also be an induced subgraph that looks like Figure \ref{fig:disc}(i). 
Note that the m-separations for this subgraph imply that $X_a \indep X_b$ and 
$X_a \indep X_c \mid X_v$;
a distribution over binary variables that satisfies both of these constraints would be 
to have $X_a + X_b + X_c =^2 0$ (where $=^2$ denotes equality modulo 2), and 
$P(X_v = 0) = 1$.  However, there is also an e-separation constraint between 
$a$ and $\{b,c\}$ if we delete $v$, and the corresponding inequality constraint is 
\emph{not} satisfied by this distribution.  Hence, there is indeed a non-trivial 
inequality.
%
\end{proof}

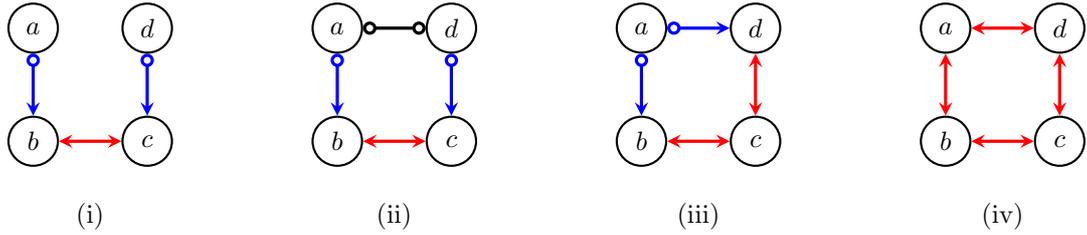
\begin{figure}
    \centering
\begin{tikzpicture}[node distance=15mm, >=stealth]
 \pgfsetarrows{latex-latex}
\begin{scope}
\node[rv] (1) {$a$};
\node[rv, below of=1] (2) {$b$};
\node[rv, right of=2] (3) {$c$};
\node[rv, right of=1] (4) {$d$};
\draw[cdeg] (1) to (2);
\draw[cdeg] (4) to (3);
\draw[beg] (2) to (3);
\node[below of=2, xshift=7.5mm, yshift=5mm] {(i)};
\end{scope}
\begin{scope}[xshift=4cm]
\node[rv] (1) {$a$};
\node[rv, below of=1] (2) {$b$};
\node[rv, right of=2] (3) {$c$};
\node[rv, right of=1] (4) {$d$};
\draw[cceg] (1) to (4);
\draw[cdeg] (1) to (2);
\draw[cdeg] (4) to (3);
\draw[beg] (2) to (3);
\node[below of=2, xshift=7.5mm, yshift=5mm] {(ii)};
\end{scope}
\begin{scope}[xshift=8cm]
\node[rv] (1) {$a$};
\node[rv, below of=1] (2) {$b$};
\node[rv, right of=2] (3) {$c$};
\node[rv, right of=1] (4) {$d$};
\draw[cdeg] (1) to (4);
\draw[cdeg] (1) to (2);
\draw[beg] (4) to (3);
\draw[beg] (2) to (3);
\node[below of=2, xshift=7.5mm, yshift=5mm] {(iii)};
\end{scope}
\begin{scope}[xshift=12cm]
\node[rv] (1) {$a$};
\node[rv, below of=1] (2) {$b$};
\node[rv, right of=2] (3) {$c$};
\node[rv, right of=1] (4) {$d$};
\draw[beg] (1) to (4);
\draw[beg] (1) to (2);
\draw[beg] (4) to (3);
\draw[beg] (2) to (3);
\node[below of=2, xshift=7.5mm, yshift=5mm] {(iv)};
\end{scope}
\end{tikzpicture}
    \caption{Up to symmetry, the four possible induced subgraphs of a PAG containing 
    a locally unshielded collider path of length 3 from $a$ to $d$.}
    \label{fig:bell}
\end{figure}

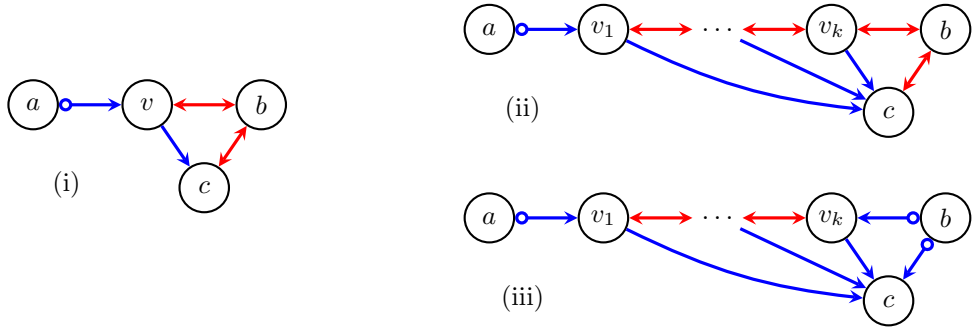
\begin{figure}
    \centering
\begin{tikzpicture}[node distance=15mm, >=stealth]
 \pgfsetarrows{latex-latex}
\begin{scope}
\node[rv] (1) {$a$};
\node[rv, right of=1] (2) {$v$};
\node[rv, right of=2] (3) {$b$};
\node[rv, below of=3, xshift=-7.5mm, yshift=4mm] (4) {$c$};
\draw[cdeg] (1) to (2);
\draw[beg] (4) to (3);
\draw[beg] (2) to (3);
\draw[deg] (2) to (4);
\node[below left of=2, xshift=0mm, yshift=0mm] {(i)};
\end{scope}
\begin{scope}[xshift=6cm, yshift=1cm]
\node[rv] (1) {$a$};
\node[rv, right of=1] (2) {$v_1$};
\node[right of=2] (2a) {$\ldots$};
\node[rv, right of=2a] (2b) {$v_k$};
\node[rv, right of=2b] (3) {$b$};
\node[rv, below of=3, xshift=-7.5mm, yshift=4mm] (4) {$c$};
\draw[cdeg] (1) to (2);
\draw[beg] (2) to (2a);
\draw[beg] (2b) to (2a);
\draw[beg] (2b) to (3);
\draw[deg] (2) to[bend right=10] (4);
\draw[deg] (2a) to[bend right=0] (4.150);
\draw[deg] (2b) to (4);
\draw[beg] (4) to (3);
\node[below left of=2, xshift=0mm, yshift=0mm] {(ii)};
\end{scope}
\begin{scope}[xshift=6cm, yshift=-1.5cm]
\node[rv] (1) {$a$};
\node[rv, right of=1] (2) {$v_1$};
\node[right of=2] (2a) {$\ldots$};
\node[rv, right of=2a] (2b) {$v_k$};
\node[rv, right of=2b] (3) {$b$};
\node[rv, below of=3, xshift=-7.5mm, yshift=4mm] (4) {$c$};
\draw[cdeg] (1) to (2);
\draw[beg] (2) to (2a);
\draw[beg] (2b) to (2a);
\draw[deg] (2b) to (4);
\draw[deg] (2) to[bend right=10] (4);
\draw[deg] (2a) to[bend right=0] (4.150);
\draw[cdeg] (3) to (4);
\draw[cdeg] (3) to (2b);
\node[below left of=2, xshift=0mm, yshift=0mm] {(iii)};
\end{scope}
\end{tikzpicture}
    \caption{Discriminating paths for $b$: (i) a path of length 3 (containing a bidirected edge); and (ii)--(iii) two 
    possible configurations of length $k+2$.}
    \label{fig:disc}
\end{figure}






The proof technique used for the graphs in Figures \ref{fig:bell}(iii) and (iv) is known 
as the `Fritz trick'\footnote{This is a term
coined by members of the Perimeter Institute, including Elie Wolfe.}, because it is a 
generalization of the approach that Tobias Fritz uses in Proposition 2.13 of
\citet{fritz12beyond}. 

\subsection{Proof of the main result}

We now have enough information to prove our main result.

\begin{proof}[Proof of Theorem \ref{thm:main}]
From the results in Section \ref{sec:nested} we know that if there is a non-trivial
nested constraint, then the set of ordinary independences induced after a final fixing
are not ones that can be represented faithfully by a DAG model.  
    
Then for such models, as well as other models without nested conditional independences, 
we can always represent the conditional independence structure by a partial ancestral graph.
If there is a necessary bidirected edge then this induces a non-trivial inequality constraint 
(Proposition \ref{prop:bidi_ineq}).  Since Proposition \ref{prop:zhang} 
tells us that  the presence of a bidirected edge in the PAG implies there is no DAG that can 
represent the equivalence class, this proves that not having a marginal model that is not Markov 
equivalent to a DAG implies the existence of a non-trivial inequality.

For the converse the result is trivial, since DAG models are defined by the finite 
list of independences in Definition \ref{dfn:lmp}.
\end{proof}

Now we have proven our main result.  Marginal DAG models can be categorized into several classes: 
(i) those which are distributionally equivalent to a DAG (Figures \ref{fig:equiv}(i)--(ii));
(ii) those with additional inequality constraints only (Figures \ref{fig:equiv}(iv));
and (iii) graphs with non-DAG-like conditional independences (Figure \ref{fig:bell}) or
(iv) graphs with nested conditional independences (Figure \ref{fig:exm}(i)), both of 
which induce inequalities.




\section{Extension to the continuous case} \label{sec:cont}

One obvious question for an extension to this paper is to ask whether or not
the result also holds in the case of variables that are not discrete. 
Bell inequalities (i.e.~ones analogous to the CHSH inequality) are known to hold 
even if all the variables are continuous \citep{cavalcanti07bell}, and indeed hold 
on arbitrary discretizations of such variables.  

However, there are obstacles to generalizing this result to the continuous case. 
The first is that the results of \citet{evans18complete} only apply to models where all
the observed variables are discrete.  Another is that results of \citet{rosset18universal} and \citet{duarte23automated} 
enable one to show that the model is semi-algebraic if observed variables have a finite 
state-space, so for continuous (or even countably infinite) state-spaces we would need 
an analogous condition.  
The final problem is that e-separation results require the distribution 
of the variables deleted to have at least one atom, even if the other variables are 
continuous.  Indeed, it is an open question 
whether inequalities are contained in models such as the one induced by the mDAG in 
Figure \ref{fig:disc}(i) when $X_v$ is continuous.  The Shannon-cone of this 
model does not induce any non-trivial entropic inequalities in that case, for example 
\citep{chaves14}. 

\subsection*{Acknowledgements}

We thank Richard Guo for suggesting the problem and Elie Wolfe for conjecturing 
the result.  This work was largely completed while the author was a visiting 
researcher at the Simons Institute in Berkeley, California.  We are also grateful
to two anonymous referees for very helpful suggestions and comments.

\bibliographystyle{abbrvnat}
\bibliography{refs}

\appendix

\section{Definitions for mDAGs}






\subsection{Basic definitions and m-separation}\label{sec:msep}

Let $\G=(V,\mathcal{D},\mathcal{B})$ be a mixed (hyper-)graph with 
directed edges $\mathcal{D}$ and bidirected simplicial complex $\mathcal{B}$.

\begin{dfn}
A \emph{path} in $\G$ is a sequence of edges and (distinct) vertices 
$\langle v_0, e_1, v_1, e_2, \ldots, e_k, v_k\rangle$, such that $v_{i-1},v_i \in e_i$ for $i=1,\ldots,k$.
A path is \emph{directed} if each $e_i$ is $v_{i-1} \to v_i$.  The
\emph{length} of the path is $k$ (the number of edges in it), and
this can be zero. 
\end{dfn}

\begin{dfn}
Given a vertex $v \in V$ in an mDAG $\G$ we define
\begin{align*}
\pa_{\G}(v) &= \{w: w \to v \text{ in } \G\}\\
\an_{\G}(v) &= \{w: w \to \cdots \to v \text{ in } \G \text{ or } w=v\}\\
\text{and} \qquad \dec_{\G}(v) &= \{w: v \to \cdots \to w \text{ in } \G \text{ or } w=v\}
\end{align*}
to be respectively the \emph{parents}, \emph{ancestors} and
\emph{descendants} of $v$.  

We use $v \lrarr w$ as a shorthand to denote that 
$v$ and $w$ are contained within some bidirected facet.  Then define
\begin{align*}
\sib_{\G}(v) &= \{w : w \lrarr v \text{ in } \G\}\\
\text{and} \qquad \dis_{\G}(v) &= \{w: w \lrarr \cdots \lrarr v \text{ in } \G \text{ or } w=v\}
    \end{align*}
to be the \emph{siblings} and \emph{district} of $v$ respectively. 
Siblings of $v$ are vertices for which a latent `parent' is shared, and the 
districts are easily identified as maximal connected red components in the 
graph. 
\end{dfn}

\begin{dfn}
Given a path $\pi$ of length $k$, an internal vertex $v_i$ (i.e.~not $v_0$ or 
$v_k$) is said to be a 
\emph{collider} on the path if the adjacent edges $e_{i},e_{i+1}$ have 
arrowheads at $v_i$.  Otherwise an internal vertex is a \emph{non-collider}.

A path from $a$ to $b$ is said to be \emph{open} given a set $C$ if no 
non-colliders on the path are in $C$, and any collider is in the set of 
vertices that can reach $C$ via a directed path (possibly of length zero).  
Otherwise the path is \emph{blocked}.

We say that sets of vertices $A$ and $B$ are \emph{m-separated} given a
set $C$ if every path from any $a \in A$ to any $b \in B$ is blocked by
$C$.  We denote this by $A \perp_m B \mid C$.
\end{dfn}

\subsection{Random variables and constraints} \label{sec:rvci}

We consider random variables $X_V = (X_v)_{v \in V}$ taking values in a finite-dimensional 
Cartesian product space $\mathcal{X}_V := \times_{v \in V} \mathcal{X}_v$.

\begin{dfn}
A distribution $p$ is said to satisfy the \emph{global Markov property} for 
an mDAG $\G$ if whenever $A,B,C$ are disjoint subsets of the vertices
of $\G$ and $A \perp_m B \mid C$, we have the corresponding 
conditional independence $X_A \indep X_B \mid X_C$ under $p$.
\end{dfn}

We can extend m-separation to \emph{e-separation} (or \emph{\textbf{e}xtended 
m-separation}) by first deleting some variables and their incident edges, and 
then checking for m-separations among what remains.

\begin{dfn}
We say sets of vertices $A$ and $B$ are \emph{e-separated} given a set $C$
and after deletion of $D$ if every path from any $a \in A$ to any $b \in B$ is 
either blocked by $C$ or passes through a node in $D$.  We denote this by 
$A \perp_e B \mid C \del D$.
\end{dfn}

Then a result from \citet{evans:12} tells us that an e-separation will induce
(at least) an inequality constraint on $p$.

\begin{thm}
Suppose that a distribution $p$ lies in the marginal model of an mDAG $\G$, and 
that the e-separation $A \perp_e B \mid C \del D$ holds in $\G$, where $X_D$ takes
values in a finite set.  Then, for every $x_D \in \mathcal{X}_D$, we have that
there exists a distribution $p^{x_D}$ such that:
\begin{align*}
\begin{array}{ll}
    p(y_{V \setminus D}, x_D) = p^{x_D}(y_{V \setminus D}, x_D) & \text{ for all } \; y_{V \setminus D} \in \mathcal{X}_{V \setminus D},
    \end{array}
\end{align*}
and $X_A \indep X_B \mid X_C$ under $p^{x_D}$.
\end{thm}

If we consider a distribution in which $X_D = x_D$ for some arbitrary state 
$x_D \in \mathcal{X}_D$ with very high probability, it is clear that this induces 
(at least) an inequality constraint.  See \citet{evans:12} for further details.

\subsection{Latent projection} \label{sec:lat_proj}

Given an mDAG $\G$ with vertices $V \dot\cup L$ where $V$ and $L$ are 
disjoint, the \emph{latent projection} of $\G$ over $V$ is given by 
the mDAG with vertices $V$ and edges within $V$ given by:
\begin{itemize}
    \item $a \to b$ whenever there is a directed walk in $\G$ from $a$ to $b$ 
    and any other (internal) vertices on the path are in $L$;
    \item $B$ is a bidirected face if there exists a \emph{source} such that 
    there is a directed path from the source down to each $b \in B$ and every
    variable on that path (other than $b$) is in $L$.
\end{itemize}
Here a `source' is either a single bidirected face or a variable that is contained 
in $L$.  See \citet{evans16mdags} for some examples.

\subsection{Maximal Ancestral Projection for mDAGs} \label{sec:anc_proj}

For this section we consider only ordinary mixed graphs (i.e.~without 
any hyper-edges) that contain both bidirected and directed edges. 

\begin{dfn}
An ordinary mixed graph is \emph{ancestral} if its directed part is
acyclic, and no vertex is an ancestor of any of its siblings; it is
\emph{maximal} if every pair of vertices that are not adjacent satisfy an 
m-separation or a nested constraint.  Note that ancestral graphs are, by definition, simple.

For an mDAG $\G$, the \emph{maximal ancestral projection} $\G^*$ includes edges 
\begin{itemize}
    \item $a \to b$ if $a \in \an_\G(b)$; and 
    \item $a \lrarr b$ if there is no ancestral relation in $\G$;
\end{itemize}
for any pair of vertices $a,b$ that cannot be m-separated in $\G$. 
\end{dfn}

The crucial fact about a maximal ancestral projection is that it always
induces precisely the same m-separations as the original mDAG did \citep{richardson02ancestral, evans16mdags}.

\subsection{Partial Ancestral Graphs} \label{sec:pag}

Given the maximal ancestral projection of an mDAG, one can consider all
these projections for all mDAGs over the same set of vertices that are
ordinary Markov equivalent to one another.  We can denote this equivalence
class $[\G]$.  Then the \emph{partial ancestral graph} $\mathcal{P} = [\G]$
is the unique graph that:
\begin{itemize}
    \item has the same skeleton as the maximal ancestral projection of any element of $[\G]$;
    \item has an arrowhead (respectively tail) in any position for which the maximal ancestral
    projection of every element of the equivalence class has an arrowhead (resp.~tail);
    \item has a circle at the end of any other edge.
\end{itemize}
More details about PAGs can be found in \citet{richardson03causal} and 
\citet{zhang:thesis,zhang08completeness}.

\subsection{Nested Models and Fixing} \label{sec:fix}

\begin{dfn}
A vertex is said to be \emph{fixable} if it has no (strict) descendants within 
its own district; that is, if $\dec_\G(v) \cap \dis_\G(v) = \{v\}$.
\end{dfn}

Note that a vertex $v$ is fixable in $\G$ precisely when, given a distribution 
$p$ that is nested Markov with respect to $\G$, we can identify the distribution
that would result if we \emph{intervened} to fix the value of $X_v = x_v$ 
from $p$ \citep{richardson23nested}.





\begin{figure}
     \centering
 \begin{tikzpicture}
 [>=stealth, node distance=20mm]
 \pgfsetarrows{latex-latex};
 \begin{scope}
 \node (0) {};
  \node[rv] at (150:1) (1a) {$a$};
   \node[rv] at (90:2) (2a) {$b$};
  \node[rv] at (30:1) (4a) {$d$};
  \node[rv] at (270:1) (5a) {$e$};
  \node[rv] at (210:2) (3a) {$c$};
  \node[rv] at (330:2) (6) {$f$};
  \draw[deg] (2a) -- (4a);
  \draw[deg] (1a) -- (2a);
  \draw[deg] (5a) -- (6);
\draw[beg] (1a) -- (3a);
  \draw[beg] (4a) to[bend left] (6);
 \draw[beg] (1a) to[bend left] (2a);
 \end{scope}
 \end{tikzpicture}
     \caption{A conditional mDAG obtained by fixing $e$ from Figure \ref{fig:mDAG}(a).} 
     \label{fig:mDAGfix}
 \end{figure}
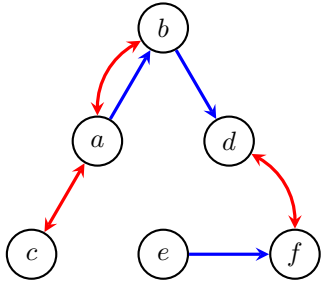

\begin{dfn}
Given an mDAG $\G$ and a vertex $v$ that is fixable, the \emph{Markov 
blanket} of $v$ is given by
\begin{align*}
    \mbl_\G(v) := (\dis_\G(v) \setminus \{v\}) \cup \pa_\G(\dis_\G(v)).
\end{align*}
\end{dfn}

For an arbitrary set $V$, let $\mathcal{P}(V)$ denote the \emph{power set}
of $V$; that is, the collection of all subsets of $V$.

\begin{dfn}
Let $\G=(V,\mathcal{D},\mathcal{B})$ be an mDAG.  Then if we can \emph{fix} 
a vertex 
$v \in V$ we obtain a new graph $\G^*$ with vertices $V$, and edges obtained by 
taking precisely those edges in $\mathcal{B} \cap \mathcal{P}(V \setminus \{v\})$ and $\mathcal{D} \cap \left(V \times (V \setminus \{v\})\right)$.  
\end{dfn}

In other words, when we fix we remove (or reduce) any edges that have 
arrowheads at the vertex that has been fixed.

\begin{dfn}
We also associate a fixing operation to the distribution.  If we fix $v$
from $\G$, then we replace $p$ with $p^*$, 
given by
\begin{align*}
p^*(x_V) &= \frac{p^*(x_v)}{p(x_v \cmid x_{\mbl(v)})} \cdot p(x_V).
\end{align*}
In other words, we remove any dependence of $X_v$ on its Markov blanket.
\end{dfn}


\begin{exm}
Consider the mDAG in Figure \ref{fig:mDAG}(i) and notice that $e$ is
fixable; after fixing it we obtain the graph in Figure \ref{fig:mDAGfix}.  
Whereas previously there was no set that could m-separate $b$ and $f$, in 
spite of them not being adjacent, notice that now they are m-separated 
conditionally
upon $e$.  This is an example of a non-trivial \emph{nested constraint}.
\end{exm}

\subsubsection*{Results relating to the nested model}

Let the set of distributions that are ordinary Markov with respect to an
mDAG $\G$ be denoted $\mathcal{O}(\G)$, and those that are nested Markov 
be denoted $\mathcal{N}(\G)$. 

\begin{prop} \label{prop:subequiv}
Suppose that $\G$ is an mDAG.  Then:
\begin{align*}
    p \in \mathcal{M}(\G) \implies p \in \mathcal{N}(\G) \implies p \in \mathcal{O}(\G).
\end{align*}
In other words, distributional equivalence is a stronger requirement than
nested equivalence, which is in turn a stronger requirement than 
ordinary equivalence.
\end{prop}

We now provide some results that are used in the proof of Proposition 
\ref{prop:noDAGm}.

\begin{lem} \label{lem:order}
Suppose that in an mDAG $\G$ we have $a \perp_m b \mid D$ but $a \not\perp_m b \mid D \cup \{s\}$.
Then there is a valid topological ordering in which $s$ comes after $a$, $b$ and every
element in $D$.
\end{lem}

\begin{proof}
    Suppose not.  Then there is a path from $a$ to $b$ that is blocked by $D$ but becomes 
    open when we also condition on $s$.  This implies that there is a collider that has $s$
    as a descendant, but no other element of $D$.  (If there are multiple colliders, then reduce 
    to one by taking the directed path from the first collider and the final collider to $s$, and 
    use whichever vertex is the one at which these paths meet.)  By the supposition that
    $a \perp_m b \mid D$ there is no directed path from $s$ to any element of $D$.

    Now, if $s$ is an ancestor of $b$ we can take the path from $a$ to the collider, then follow
    the directed path from here to $s$ and then to $b$.  Clearly this path is open without 
    conditioning on $s$, so we reach a contradiction. 
\end{proof}

\begin{prop} \label{prop:noDAG}
Consider an independence model $\mathcal{I}$ such that:
\begin{align*}
&& && v &\indep s &&[\mathcal{I}] && &&\\
&& && a &\indep b \mid D &&[\mathcal{I}] && &&\\
&& && a &\not\indep b \mid D \cup \{s\} &&[\mathcal{I}], && &&
\end{align*}
where $v \in \{a,b\}$.  Then there is no DAG that 
faithfully represents the independence model $\mathcal{I}$.

If $D$ is chosen to be inclusion minimal such that $a \indep b \mid D$ holds,
then $v \in D$ is also not allowed by any faithful DAG independence model.
\end{prop}

\begin{proof}
From Lemma \ref{lem:order} we know that none of $a$, $b$ or $D$ are necessarily 
descendants of $s$.  In this case, choose a particular DAG such that $s$ comes 
after $a$, $b$ and $D$ in the chosen topological ordering (say $<$).  

Then the only way in which $\mathcal{I}$ could hold with a factorization that 
represents a DAG is if we can divide the predecessors of $s$ under $<$ into two sets $S \cup T$, and
we have $S \cup \{s\} \perp_m T$, with $v \in T$.  In this case, if
either $a$ or $b$ is in $T$ then conditioning on $s$ cannot make them 
dependent conditional on any subset that m-separates them.  If $a,b \notin T$
but some $d \in D \cap T$, 
then the m-separation between $a$ and $b$ would hold given $D \setminus \{d\}$, 
which contradicts the minimality of $D$.  Either way, we obtain the result. 
%
%
%
%
\end{proof}

\section{Other results} \label{sec:or}


\begin{prop} \label{prop:disc_path2}
Suppose that there is an mDAG with no locally unshielded collider path of 
length at least 3, but that does have a discriminating path from $a$ to $c$ 
for $b$ of length at least 4.  Then there also exists an induced subgraph 
isomorphic to Figure \ref{fig:disc}(i).
\end{prop}

\begin{proof}
If there is a discriminating path $\langle a,v_1,\ldots,v_k,b,c\rangle$ 
with $k \geq 2$, then clearly either there is a locally unshielded 
collider path of length at least 3, or there is 
a directed edge between the vertices $v_{k-1}$ and $b$, or between $v_{k-2}$ and $v_k$.  In the latter case, the colliders 
$\langle v_{k-2},v_{k-1},v_k\rangle$ and 
$\langle v_{k-1},v_k,b \rangle$ must have discriminating paths of a 
strictly lower order of their own \citep{claassen22greedy}.  Hence
we can consider a lower order discriminating path, and by induction we 
will eventually reach a first-order discriminating path.  In this case, 
if $k \geq 2$ then $a \, \sto v_1 \lrarr v_2 \getss \, v_3$ (where 
possibly $v_3 = b$) will be a locally unshielded collider path of 
length 3, or we will have a discriminating path that looks like 
Figure \ref{fig:disc}(i).
%
%
%
\end{proof}

\end{document}